\documentclass[letterpaper, 10pt, conference]{ieeeconf}      

\IEEEoverridecommandlockouts                              
\usepackage{amsmath,graphicx,amsfonts,amssymb,epsfig,mathrsfs,enumerate}
\usepackage{subcaption}
\usepackage{color}
\usepackage{multirow}
\usepackage{rotating}
\usepackage{graphicx}%
\usepackage{algorithm}
\usepackage{algpseudocode}
\usepackage{algorithmicx}
\usepackage{cite}
\usepackage{url}
\usepackage{framed}
\usepackage{bm}

\newtheorem{theorem}{Theorem}

\newtheorem{definition}{Definition}

\newtheorem{proposition}{Proposition}
\newtheorem{remark}{Remark}

\newtheorem{example}{Example}

\usepackage{tikz}
\usetikzlibrary{calc,trees,positioning,arrows,chains,shapes.geometric,decorations.pathreplacing,decorations.pathmorphing,shapes, matrix,shapes.symbols}

\newcommand{\bpi}{{\bm{\pi}}}

\newcommand{\bg}{{\bm{g}}}

\newcommand{\bh}{{\bm{h}}}
\newcommand{\bw}{{\bm{w}}}
\newcommand{\bx}{{\bm{x}}}
\newcommand{\by}{{\bm{y}}}
\newcommand{\bz}{{\bm{z}}}
\newcommand{\bbf}{{\bm{f}}}
\newcommand{\bv}{{\bm{v}}}
\newcommand{\bu}{{\bm{u}}}

\newcommand{\bgamma}{{\bm{\Gamma}}}
\newcommand{\btau}{{\bm{\tau}}}
\newcommand{\bdelta}{{\bm{\delta}}}

\newcommand{\bA}{{\bm{A}}}
\newcommand{\bB}{{\bm{B}}}
\newcommand{\bC}{{\bm{C}}}

\newcommand{\bD}{{\bm{D}}}

\newcommand{\bG}{{\bm{G}}}
\newcommand{\sopt}{\sigma^{{\rm{opt}}}}
\newcommand{\vopt}{\bv^{{\rm{opt}}}}
\newcommand{\ropt}{\rho^{{\rm{opt}}}}
\newcommand{\uopt}{\bu^{{\rm{opt}}}}

\renewcommand{\dagger}{{T}}
\newcommand{\differential}{{\rm{d}}}

\newcommand{\spt}{{\mathrm{spt}}}

\allowdisplaybreaks

\title{\LARGE\textbf{
Finite Horizon Density Steering for Multi-input\\
State Feedback Linearizable Systems}
}

\author{Kenneth F. Caluya, Abhishek Halder
\thanks{Kenneth F. Caluya, and Abhishek Halder are with the Department of Applied Mathematics, University of California, Santa Cruz, CA 95064, USA,
        {\tt\small{\{kcaluya,ahalder\}@ucsc.edu}}%
}\\
\thanks{This research was partially supported by NSF award 1923278.}}

\begin{document}

\maketitle
\thispagestyle{empty}
\pagestyle{empty}

\def\spacingset#1{\def\baselinestretch{#1}\small\normalsize}
\setlength{\parindent}{20pt}
\setlength{\parskip}{12pt}
\spacingset{1}

\title{\huge{Multi-Input Feedback Linearization}}

\author{Kenneth F. Caluya, Abhishek Halder}

\markboth{\today}{}

\maketitle

\begin{abstract}
In this paper, we study the feedback synthesis problem for steering the joint state density or ensemble subject to multi-input state feedback linearizable dynamics. This problem is of interest to many practical applications including that of dynamically shaping a robotic swarm. Our results here show that it is possible to exploit the structural nonlinearities to derive the feedback controllers steering the joint density from a prescribed shape to another while minimizing the expected control effort to do so. The developments herein build on our previous work, and extend the theory of the Schr\"{o}dinger bridge problem subject to feedback linearizable dynamics. 
\end{abstract}


\section{Introduction}
We consider the problem of steering the statistics of the state vector $\bx(t)$ from a prescribed ensemble or joint density $\rho_{0}(\bx)$ to another $\rho_{1}(\bx)$ over a finite time horizon $t\in[0,1]$, subject to controlled nonlinear dynamics of the form 
\begin{equation} \label{AffNonLine}
\begin{aligned}
    \dot{\bx} = \bbf(\bx) + \bm{G}(\bx)\bu, \quad \bx\in \mathcal{X}\subseteq \mathbb{R}^n, \quad \bu \in \mathbb{R}^m,
    \end{aligned}
\end{equation}
where $\bbf$ is a smooth vector field on the state space $\mathcal{X}\subseteq \mathbb{R}^n$, and $\bm{G}$ is an $n\times m$ matrix whose columns consist of the vectors $\bg_{i} \in \mathbb{R}^{n}$ for $i=1,\dots,m$, i.e.,
\begin{equation}
    \bm{G}(\bx) = \left[ \bg_1(\bx)|\bg_2(\bx)|\dots|\bg_m
    (\bx)\right].
\end{equation}
It is of broad practical interest to solve this finite horizon density steering problem while minimizing the average total control effort over the \emph{controlled} state ensemble $\rho(\bx,t)$. 

This problem is motivated by the growing need across science and engineering applications to control a \emph{large population of systems}. Consider for example, shaping the bulk magnetization distribution for Nuclear Magnetic Resonance spectroscopy, controlling heterogeneous (e.g., aerial and ground) robotic swarms \cite{bandyopadhyay2017probabilistic,deshmukh2018mean}, strategically synchronizing and desynchronizing a neuronal population to regulate the Parkinsonian tremor \cite{monga2018synchronizing}, and differentially moving the setpoints of a large population of residential air-conditioners by a service provider to make their total energy consumption track the intermittency in supply (e.g., due to stochastic renewable generation) in a privacy-preserving manner \cite{chertkov2017ensemble,halder2016architecture}. These exemplars concern \emph{population ensemble or density} whose shape is actively controlled over time while preserving the physical mass. The conservation of mass allows an alternative interpretation of the underlying mathematical problem -- instead of steering a large number of dynamical systems, one can think of steering a single system with probabilistic uncertainty in its initial and terminal state, i.e., $\rho_{0},\rho_{1}$ being joint state probability density functions (PDFs). This too, arises naturally in practice, e.g., in robot motion planning \cite{okamoto2019optimal}, where uncertainties in the initial and terminal states are unavoidable due to process and sensor noise.

From a system-control theoretic viewpoint, finite horizon density steering \emph{via feedback} is a non-classical stochastic optimal control problem. The qualifier ``non-classical" points to the fact that finding the feedback policy requires solving an infinite dimensional two-point boundary value problem on the manifold of joint state PDFs. This is an emerging research direction in the systems-control community wherein recent advances \cite{chen2015optimal1,chen2015optimal2,chen2017optimal,chen2016relation,halder2016finite} have uncovered its connections with the theory of optimal mass transport \cite{villani2003topics,villani2008optimal} and the Schr\"{o}dinger Bridge Problem (SBP) \cite{schrodinger1931umkehrung,schrodinger1932theorie,leonard2014survey}. Also, there have been results on the covariance steering problem \cite{hotz1987covariance,skelton1989covariance,zhu1995covariance,skelton1997unified,okamoto2018optimal} which concerns steering second order state statistics. With the exception of \cite{agrachev2009optimal,elamvazhuthi2018optimal}, almost all works have focused on steering the state statistics over a controlled linear system. 

In this paper, we consider finite horizon density steering state feedback linearizable systems of the form (\ref{AffNonLine}). The nonlinearities in (\ref{AffNonLine}) induce non-Gaussian statistics even if the endpoint PDFs are both jointly Gaussian. Thus, finding the feedback solution of the density steering problem in a non-parametric sense, is non-trivial. The main contribution of this paper is to show that it is possible to exploit the feedback linearizing transformation for density steering. In particular, we obtain the optimal state feedback policy in terms of the solution of certain Hamilton-Jacobi-Bellman (HJB) partial differential equation (PDE). Furthermore, we show that a dynamic stochastic regularization can be used to derive a system of boundary-coupled linear PDEs, which we refer to as the \emph{Schr\"{o}dinger system}, whose solutions recover the optimal state feedback and the optimal controlled joint state PDF. We envision that the theoretical developments herein will help design algorithms solving the feedback density steering over nonlinear dynamical systems.

\subsection*{Notations and Preliminaries}
Throughout the paper, we will use bold-faced
capital letters for matrices and bold-faced lower-case letters
for column vectors. We use the symbol $\langle \cdot , \cdot \rangle $ to denote the Euclidean inner product. In particular, $\langle \bA, \bB \rangle := {\rm{trace}}(\bA^{\top}\bB)$
denotes Frobenius inner product between matrices $\bA$ and
$\bB$, and $\langle \bm{a}, \bm{b} \rangle  := \bm{a}^{\top}
\bm{b}$ denotes the inner product between
column vectors $\bm{a}$ and $\bm{b}$.
We use $\bm{0}$ and $\bm{1}$ to denote the vector consisting of all zeros, and all ones, respectively, and the symbol $\bm{e}_{k}$ to denote the $k$\textsuperscript{th}  standard basis vector of appropriate dimension. We use $\nabla_{\bx}$, $\nabla_{\bx}\cdot$, and ${\rm{Hess}}(\cdot)$ to respectively denote the Euclidean gradient, divergence and Hessian operators w.r.t. vector $\bx$. The Lie bracket of two vector fields $\bm{\xi}$ and $\bm{\eta}$ at $\bx\in\mathbb{R}^{n}$ is a new vector field $\left[\bm{\xi},\bm{\eta}\right](\bx):=(\nabla_{\bx}\bm{\eta})\bm{\xi}(\bx) - (\nabla_{\bx}\bm{\xi})\bm{\eta}(\bx)$. For $k\in\mathbb{N}$, the $k$-fold Lie bracketing of $\bm{\eta}$ with the same vector field $\bm{\xi}$ is denoted as ${\rm{ad}}_{\bm{\xi}}^{k}\bm{\eta}:=\left[\bm{\xi},{\rm{ad}}_{\bm{\xi}}^{k-1}\bm{\eta}\right]$; by convention ${\rm{ad}}_{\bm{\xi}}^{0}\bm{\eta}=\bm{\eta}$. The Lie derivative of a scalar-valued function $\lambda(\bx)$ w.r.t. the vector field $\bm{\xi}$ evaluated at $\bx$ is $L_{\bm{\xi}}\lambda(\bx):=\langle\nabla_{\bx}\lambda,\bm{\xi}\rangle(\bx)$. For $k\in\mathbb{N}$, the $k$-fold Lie derivative of $h$ w.r.t. the same vector field $\bm{\xi}$ evaluated at $\bx$ is denoted as $L_{\bm{\xi}}^{k}\lambda(\bx):=\langle\nabla_{\bx}L_{\bm{\xi}}^{k-1}\lambda,\bm{\xi}\rangle(\bx)$; by convention $L_{\bm{\xi}}^{0}\lambda(\bx)=\lambda(\bx)$. Given $m$ vectors fields $\bm{\xi}_{1}(\bx), \hdots, \bm{\xi}_{m}(\bx)$ in $\mathbb{R}^{n}$, we say $\mathcal{D}(\bx):={\rm{span}}\{\bm{\xi}_{1}, \hdots, \bm{\xi}_{m}\}(\bx)$ is involutive at $\bx\in\mathbb{R}^{n}$, if for all $\bm{\xi}_{i}(\bx),\bm{\xi}_{j}(\bx)\in\mathcal{D}(\bx)$, we get that the Lie bracket $\left[\bm{\xi}_{i},\bm{\xi}_{j}\right](\bx)\in\mathcal{D}(\bx)$, where $i,j=1,\hdots,m$. The notation $\bx \sim \rho$ denotes that the random vector $\bx$ has joint PDF $\rho$. Furthermore, we denote the pushforward of a PDF by the symbol $\sharp$. We use the symbol $\circ$ to denote function composition, and $\spt(\cdot)$ to denote the support of a function.


\section{MIMO Feedback Linearization}\label{SecMIMOFeedbackLinearization}
We consider multiple-input control system of the form (\ref{AffNonLine}), and recall some well-known results on feedback linearization \cite{isidori1989nonlinear} that will be useful in the sequel.
\begin{definition} (\textbf{Full state static feedback linearization})
System (\ref{AffNonLine}) is said to be \emph{full state static feedback linearizable} around a point $\bx_0 \in \mathcal{X}$ if there exists a smooth feedback of the form $\bu = \bdelta(\bx) + \bgamma(\bx) \bv$ defined on $\mathcal{X}$, and a  diffeomorphism $\btau:\mathcal{X}\mapsto \mathbb{R}^n$ such that the change of variables $\bz := \btau(\bx)$ transforms (\ref{AffNonLine}) into 
\begin{equation} \label{LinearSystem}
\begin{aligned} 
    \dot{\bz} &= \bA \bz + \bB \bv, \quad \bz \in \mathbb{R}^n, \quad \bv \in \mathbb{R}^m,
    \end{aligned}
\end{equation}
wherein the pair $(\bA,\bB)$ is controllable.
\end{definition}
 In other words, (\ref{AffNonLine}) is full state static feedback linearizable if there exists a triple $(\bdelta(\bx),\bgamma(\bx),\btau(\bx))$ such that 
\begin{equation}
\begin{aligned}
    \left(\nabla \btau \left(\bbf(\bx)+\bm{G}(\bx)\bdelta(\bx) \right) \right)_{\bx=\btau^{-1}(\bz)} &= \bA \bz, \\ 
      \left(\nabla \btau \left(\bm{G}(\bx)\bgamma(\bx) \right) \right)_{\bx=\btau^{-1}(\bz)} &= \bB,
    \end{aligned}
\end{equation}
where the pair $(\bA,\bB)$ satisfies
\begin{equation}
    {\rm{rank}}\left[\bB,\bA\bB,\bA^2\bB,\dots\bA^{n-1}\bB\right] = n.
\end{equation} 
\begin{definition}\cite[p.~220]{isidori1989nonlinear} (\textbf{Vector relative degree})
Consider the Multi-Input Multi-Output (MIMO) system:
\begin{equation} \label{MIMOSystem}
\begin{aligned} 
    \dot{\bx} &= \bbf(\bx) + \bG(\bx) \bu, \\
    \by&= \bm{h}(\bx),
    \end{aligned}
    \end{equation}
where $\bx\in\mathcal{X}\subseteq\mathbb{R}^{n}, \bu\in\mathbb{R}^{m}$ as before. Furthermore, $\bm{h}(\bx):=\left(h_1(\bx),h_2(\bx),\dots h_m(\bx) \right) \in \mathbb{R}^{m}$, where $h_j$ are smooth scalar-valued functions for all $j=1\dots,m$. The input-output system (\ref{MIMOSystem}) is said to have \emph{vector relative degree} $\bm{\pi} = (\pi_1,\pi_2,\dots,\pi_m)$ at $\bx_0 \in \mathcal{X}$, if
\begin{equation} \label{RelDegCond1}
L_{\bg_j}L^{k}_{\bbf}h_i(\bx) \equiv 0, \quad 1 \leq i,j\leq m, \quad 1\leq k<\pi_{i}-1,
\end{equation} 
and the $m\times m$ matrix
\begin{equation} \label{RelDegCond2}
\!\!\!\bC(\bx) := \!\!\begin{bmatrix}L_{\bg_1}L^{\pi_1-1}_\bbf h_1(\bx)  & \dots  & L_{\bg_m}L^{\pi_1-1}_\bbf h_1(\bx) \\
\vdots &\ddots   & \vdots \\
L_{\bg_1}L^{\pi_m-1}_\bbf h_m(\bx)  & \dots  & L_{\bg_m}L^{\pi_m-1}_\bbf h_m(\bx) 
\end{bmatrix}
\end{equation}
evaluated at $\bx=\bx_0$, is non-singular. \\
\end{definition}
Here, $\pi_j\in\mathbb{N}$, $j=1,\hdots,m$, is the number of times one has to differentiate the $j$\textsuperscript{th} output $y_j$ w.r.t. $t$ such that \emph{at least} one of the $m$ input components appears explicitly in the expression for $y_j^{(\pi_j)}$. In other words, $\pi_j$ is the number of integrators between the input and the $j$\textsuperscript{th} output.

\begin{remark} \label{Remark1}
It is known \cite[p.~230]{isidori1989nonlinear} that given an $n$-dimensional vector field $\bbf$, and a matrix $\bG(\bx)$ of rank $m$, the system (\ref{AffNonLine}) is full state static feedback linearizable if and only if:\\
(i) there exist functions $h_1(\bx), h_2(\bx),\dots,h_m(\bx)$, such that the input-output system (\ref{MIMOSystem}) has relative degree $\bpi$ at $\bx_0\in \mathcal{X}$,\\
\emph{and}\\ 
(ii) the relative degree $\bpi$ is such that $\pi_1+\pi_2+\dots + \pi _m =n$, where $n$ is the dimension of the state vector $\bx$.
\end{remark}
The output function $\bh(\bx)$ play an  important role in transforming (\ref{AffNonLine}) into a controllable linear system. If we can find $\bh(\cdot)$ satisfying conditions (i)-(ii) in Remark \ref{Remark1}, then we can use the same to construct a state feedback law and a desired change of coordinates. Explicitly, this feedback law can be obtained as 
\begin{equation} \label{FeedBackLaw}
    \bu = \underbrace{-(\bC(\bx))^{-1}\bm{d}(\bx)}_{:=\bdelta(\bx)}+ \underbrace{(\bC(\bx))^{-1}}_{:={\bgamma(\bx)}} \bv,
\end{equation}
where $\bC(\bx)$ is as in (\ref{RelDegCond2}), and 
\begin{equation}\label{Defd}
    \bm{d}(\bx):= \left(L^{\pi_1}_{\bbf}h_1(\bx),\:L^{\pi_2}_{\bbf}h_2(\bx),\:\dots,\:L^{\pi_m}_{\bbf}h_m(\bx)\right)^{\!\top},
\end{equation}
The linearizing coordinates $\bz := \btau (\bx)$ are subdivided as
\begin{equation}\label{subvectransform}
    \bz = \begin{pmatrix}\bz^{1}
    \\ \bz^{2}\\ \vdots \\ \bz^{m}  \end{pmatrix},\quad     \btau (\bx) = \begin{pmatrix}\btau ^{1}(\bx)
    \\ \btau ^{2} (\bx)\\ \vdots \\ \btau ^{m} (\bx) \end{pmatrix},
\end{equation}
where each $\bz^{i},\btau^{i} \in \mathbb{R}^{\pi_{i}}$, $i=1,\hdots,m$,  have  components 
\begin{equation}\label{DefTransform}
    z^{i}_{k} = \tau^{i}_{k}(\bx):=L_{\bbf}^{k-1}h_{i}(\bx), \quad k=1,\hdots,\pi_{i}.
\end{equation}
The feedback law (\ref{FeedBackLaw}) and the change of coordinates (\ref{subvectransform}), together transform (\ref{AffNonLine}) which is in state-control pair $(\bx,\bu)$, into the
Brunovsky canonical form in the state-control pair $(\bz,\bv)$, given by
\begin{equation} \label{LinearizedODE}
    \dot{\bz} = \bA \bz + \bB \bv, \quad \bz\in\mathbb{R}^{n}, \quad \bv\in\mathbb{R}^{m},
\end{equation}
where $\bA,\bB$ are \emph{block diagonal} matrices
\begin{equation*}
    \bA := {\rm{diag}}(\bA_1,\bA_2, \dots,\bA_m), \quad \bB:= {\rm{diag}}(\bm{b}_1,\bm{b}_2,\dots,\bm{b}_m),
\end{equation*} 
wherein for each $i=1,\hdots,m$, we have 
\begin{equation*}
    \bA_i := [\bm{0}|\bm{e}_1|\bm{e}_2|\dots|\bm{e}_{\pi_{i-1}}] \in \mathbb{R}^{{\pi_{i}\times \pi_{i}}}, \quad \bm{b}_i:= \bm{e}_{\pi_i} \in \mathbb{R}^{\pi_{i}}.
\end{equation*}

\begin{remark}
Since static state feedback linearization is equivalent to Remark \ref{Remark1}, hence the matrix (\ref{RelDegCond2}) is invertible at $\bx=\bx_0$. This guarantees that $\bgamma(\bx)$ and $\bdelta(\bx)$ in (\ref{FeedBackLaw}) are well-defined at $\bx_0$.
\end{remark}
As seen above, the existence of the (fictitious) output $\bh(\cdot)$ is a necessary and sufficient condition for full state feedback linearization. The following result allows us to establish the existence of the $\bh(\cdot)$ under suitable conditions on the vector fields $\bbf(\bx),\bg_1(\bx),\dots, \bg_m(\bx)$. Thus, the  conditions for full state feedback linearization can be restated as the following.
\begin{proposition}\label{PropositionCheckFeedbackLinearizable}
\cite[p.~232]{isidori1989nonlinear}
Consider the system (\ref{AffNonLine}) where ${\rm{rank}}(\bG(\bx_0))= m$, and for $i = 0,1,\dots, n-1$, let 
    \[\Delta_{i}(\bx) :=  {\rm{span}} \{ {\rm{ad}}_{\bbf}^{k} \bg_{k}: 0 \leq k \leq i,1\leq j \leq m\}.\]
Then, there exist scalar-valued functions $h_1(\bx),\dots, h_m(\bx)$, defined on $\mathcal{X}$ such that (\ref{MIMOSystem}) has relative degree $\bm{\pi}$ at $\bx_0$, with $\pi_1+ \pi_2 + \dots + \pi_m = n$, iff:\\
(i) $\Delta_i$ has constant dimension near $\bx_0$ for each $i=0,1,\hdots,n-1$,\\
(ii) $\Delta_{n-1}$ has dimension $n$,\\
(iii) $\Delta_i$ is involutive for each $i=0,1,\hdots,n-2$.
\end{proposition}
Proposition \ref{PropositionCheckFeedbackLinearizable} helps to verify if a given system of the form (\ref{AffNonLine}) is full-state feedback linearizable. For the construction of the functions $h_i$ in Proposition 1, we refer the readers to \cite{isidori1989nonlinear}. 
\begin{example}
Let us consider a system of the form (\ref{AffNonLine}) defined on a neighborhood of $\bx_0=\bm{0}$, given by
\begin{align}
\dot{\bx} &= \underbrace{\begin{pmatrix} x_2+x_2^2\\ x_3-x_1x_4+x_4x_5 \\ x_2x_4+x_1x_5-x_5^2 \\ 
x_5 \\ x_2^2 
\end{pmatrix}}_{\bbf(\bx)} +\underbrace{ \begin{pmatrix} 0 \\ 0 \\ \cos(x_1-x_5) \\ 0  \\ 0
\end{pmatrix}}_{\bg_1(\bx)}u_1 \nonumber\\
& \qquad \qquad + \underbrace{\begin{pmatrix} 1 \\ 0 \\ 1 \\0 \\ 1
\end{pmatrix}}_{\bg_2(\bx)} u_2.
\label{MIMOExample}	
\end{align}
A direct computation verifies that (\ref{MIMOExample}) satisfies the conditions (i)-(iii) of Proposition \ref{PropositionCheckFeedbackLinearizable}, implying the existence of output functions $h_1(\bx),h_2(\bx)$. Following the constructive steps in \cite[p.~232]{isidori1989nonlinear}, these output functions can be obtained as
\begin{equation} \label{OutputsExample}
    y_1 =h_1= x_1 -x_5, y_2 =h_2=x_4. 
\end{equation}
\end{example}
Notice that 
\begin{equation*}
\begin{aligned}
    L_{\bg_1} h_1(\bx) & \equiv L_{\bg_2} h_1(\bx) \equiv L_{\bg_1}L_{\bbf} h_1(\bx) \equiv L_{\bg_2}L_{\bbf} h_1(\bx) \equiv 0,\\ 
    L_{\bg_1} h_2(\bx) &\equiv  L_{\bg_2} h_2(\bx) \equiv 0, 
\end{aligned}
\end{equation*}
and that the matrix
\begin{equation*}
 \begin{aligned}
 \bC(\bx_0) &=
\begin{pmatrix}
     L_{\bg_1}L_{\bbf}^2 h_1(\bx) &   L_{\bg_2}L_{\bbf}^2 h_1(\bx)  \\ 
       L_{\bg_1}L_{\bbf} h_2(\bx) &
         L_{\bg_2}L_{\bbf} h_2(\bx) 
\end{pmatrix} \bigg\vert_{\bx=\bx_0}\\
&= 
\begin{pmatrix}
\cos(x_1-x_5) & 1 \\ 0 & 1
\end{pmatrix}\bigg |_{\bx=\bx_0} =\begin{pmatrix}
1 & 1 \\ 0 & 1
\end{pmatrix},
\end{aligned}
\end{equation*}
is non-singular. Therefore, (\ref{MIMOExample}) with output (\ref{OutputsExample}) has vector relative degree $\bpi =(\pi_1,\pi_2)= (3,2)$ satisfying $\pi_1+\pi_2 = 3+2 =5$, which is indeed the dimension of the state space. From (\ref{DefTransform}), we obtain the change of coordinates
\begin{equation}\label{ExampleChangeofCoordinates}
    \btau(\bx):= \begin{pmatrix} 
h_1(\bx) \\ L_{\bbf}h_1(\bx) \\  L_{\bbf}^2 h_1(\bx) \\ h_2(\bx) \\ L_{\bbf} h_2(\bx) 
    \end{pmatrix}= \begin{pmatrix}
x_1-x_5 \\ x_2 \\ x_3 -x_1x_4+x_4x_5 \\ x_4 \\ x_5
\end{pmatrix}.
\end{equation}
In this case, (\ref{RelDegCond2}) and (\ref{Defd}) yields
\begin{equation*}
\begin{aligned}
    \bm{d}(\bx)& = \begin{pmatrix} 
    L^{3}_{\bbf} h_1(\bx) \\ L^{2}_{\bbf} h _2(\bx)  
    \end{pmatrix} = \begin{pmatrix}
    0 \\ x_2^2
    \end{pmatrix}, \\ 
    (\bC(\bx))^{-1} 
    &= \begin{pmatrix}
    1/\cos(x_1-x_5) &  -1/\cos(x_1-x_5) \\ 0  & 1
    \end{pmatrix},
    \end{aligned}
\end{equation*}
which, following (\ref{FeedBackLaw}), result in the feedback law  
\begin{equation} \label{ExampleFeedbackLaw} 
\begin{aligned}
\bu &= \underbrace{\begin{pmatrix} 
- x_2^2/\cos(x_1-x_5) \\  x_2^2   
\end{pmatrix}}_{\bdelta(\bx)} \\
&+ 
\underbrace{\begin{pmatrix} 
 1/\cos(x_1-x_5)&-1 \cos(x_1-x_5)  \\ 0 & 1  
\end{pmatrix}}_{\bgamma(\bx)} \bv.
\end{aligned}
\end{equation}
Hence, we have constructed a triple $\left(\bdelta(\bx),\bgamma(\bx),\bm{\tau}(\bx)\right)$ given by (\ref{ExampleChangeofCoordinates})-(\ref{ExampleFeedbackLaw}), that transform (\ref{MIMOExample}) into
\begin{equation}\label{ExampleAB}
\dot{\bz} =\underbrace{\begin{pmatrix}
0 &  1 &  0 & 0 & 0 \\
0 &  0 &  1 & 0 & 0 \\
0 &  0&  0 & 0 & 0 \\
0 &  0 &  0 & 0 & 1 \\
0 &  0 &  0 & 0 & 0 \\
\end{pmatrix}}_{\bA} \bz +
\underbrace{\begin{pmatrix}
0 &  0  \\
0 &  0  \\
1 &  0 \\
0 &  0  \\
0 &  1  \\
\end{pmatrix}}_{\bB}\bv.
\end{equation}


\section{Minimum Energy Density Control}\label{SecMinEnergysteering}

\subsection{Stochastic Optimal Control Problem}
Given system (\ref{AffNonLine}), and two prescribed endpoint PDFs $\rho_{0}(\bx),\rho_{1}(\bx)$, we consider the following minimum energy finite horizon stochastic optimal control problem:
\begin{subequations}\label{StochasticOCP}
\begin{align} 
& \qquad\quad\underset{\bu \in \mathcal{U}} {\text{inf}}
& &  \mathbb{E} \left\{\int_{0}^{1} \frac{1}{2} \lVert \bu(x,t)  \rVert_{2}^{2}  \: \differential t \right\}, \label{StochasticOCP1}\\
& \;  \; \; \text{subject to}
& & \dot{\bx} = \bbf(\bx) + \bG(\bx) \bu, \label{StochasticOCP2}\\  
& & & \bx(0) \sim  \rho_{0}(\bx) \quad \bx(1) \sim \rho_{1}(\bx), \label{StochasticOCP3}
\end{align}
\end{subequations}
where the state space is $\mathcal{X}\subseteq \mathbb{R}^n$, $\bu \in \mathbb{R}^m$ and  (\ref{StochasticOCP2}) is \emph{feedback linearizable}. The infimum is taken over the set of admissible controls with finite energy, i.e.,  $\mathcal{U}:=\{\bu:\mathbb{R}^n \times [0,1]\mapsto \mathbb{R}^m |\: \lVert \bu \rVert_{2}^{2} < \infty \}$, and the expectation operator $\mathbb{E}\{\cdot\}$ in (\ref{StochasticOCP1}) is w.r.t. the controlled joint state PDF $\rho(\bx,t)$ satisfying endpoint conditions (\ref{StochasticOCP3}). The objective is to steer the joint PDF $\rho(\bx,t)$ from the given initial PDF $\rho_0$ at $t=0$ to a terminal PDF $\rho_1$ at $t=1$ while minimizing the expected control effort. 

The problem (\ref{StochasticOCP}) can be recast into a ``fluid dynamics" version \cite{benamou2000computational}, which is the following variational problem:
\begin{subequations} \label{FluidDynamicsVersion}
\begin{align} 
& \underset{\rho,\bu} {\inf}
& & \int_{0}^{1}\int_{\mathcal{X}} \frac{1}{2} \lVert \bu(\bx,t) \rVert_{2}^{2}\:\rho(\bx,t)  \: \differential \bx\:\differential t, \label{FluidDynamicsVersion1}\\
& \text{subject to}
& & \frac{\partial \rho }{\partial t} + \nabla_\bx \cdot(\rho (\bbf(
\bx) + \bG(\bx) \bu )) =0, \label{FluidDynamicsVersion2} \\  
& & & \rho(\bx,t=0) = \rho_{0}, \quad \rho(\bx,t=1) = \rho_{1}\label{FluidDynamicsVersion3}. 
\end{align}
\end{subequations}
Here, the infimum is taken over $\mathcal{P}(\mathcal{X}) \times \mathcal{U}$, where $\mathcal{P}(\mathcal{X}) $ denotes the space of all joint PDFs supported on $\mathcal{X}$. We note that (\ref{FluidDynamicsVersion2}) is the Liouville PDE \cite{brockett2007optimal} associated to the dynamical system (\ref{StochasticOCP2}).

\subsection{Reformulation in Feedback Linearized Coordinates}
In our recent work \cite{caluya2019finite}, we considered the problem (\ref{StochasticOCP}) for the single-input case, i.e., the case $\bG(\bx)\equiv [ \bg_1(\bx)]\in\mathbb{R}^{n}$, and the input $u$ is scalar-valued. The main idea in \cite{caluya2019finite} was to recast (\ref{FluidDynamicsVersion}), which is in state-control pair $(\bx,\bu)$, into an equivalent formulation in feedback linearized state-control pair $(\bz,\bv)$. This was made possible by using the diffeomorphism $\btau: \mathcal{X} \mapsto \mathcal{Z}$ to pushforward the endpoint PDFs $\rho_0,\rho_1$ to PDFs $\sigma_0,\sigma_1$ supported on the feedback linearized state space $\mathcal{Z}$. Specifically,
\begin{equation}
    \sigma_i(\bz):= \btau_{\sharp} \rho_i = \dfrac{\rho_i(\btau^{-1}(\bz))}{\lvert {\rm{det}}(\nabla_{\bx}\btau_{\bx=\tau^{-1}(\bz}) \rvert }, \quad i \in \{0,1\},
\end{equation}
and $\mathcal{Z} := \{\bz \in \mathbb{R}^n| \bz = \tau(\bx), \bx \in \mathcal{X} \}$.

Since $\btau$ is a diffeomorphism, 
the PDFs $\{\sigma_i\}_{i=0,1}$ supported on the feedback linearized state space $\mathcal{Z}$, are well defined, i.e., ${\rm{spt}}(\sigma_i) \subseteq \mathcal{Z}$
provided that ${\rm{spt}}(\rho_i)\subseteq \mathcal{X}$.

To generalize the reformulation in \cite[Sec. III.B]{caluya2019finite} for the multi-input case, we proceed by setting
\begin{equation}\label{defCompositeMap}
\bdelta_{\btau } : = \bdelta \circ \btau ^{-1}, \quad \bgamma_{\btau }: = \bgamma \circ \btau ^{-1}, 
\end{equation}
where $\bm{\delta}$ and $\bm{\Gamma}$ are as in (\ref{FeedBackLaw}). Using $\bu(\bz) = \bdelta_{\bm{\tau }}(\bz) + \bgamma_{\btau }(\bz)\bv$, we now transcribe (\ref{FluidDynamicsVersion}) into
\begin{subequations} \label{FluidsVersionFL}
\begin{align} 
& \underset{\sigma,\bv} {\inf}
& & \int_{0}^{1}\int_{\mathcal{Z}} \frac{1}{2} \mathcal{L}(\bz,\bv) \:\sigma(\bz,t)  \: \differential \bz\: \differential t,\label{FluidsVersionFL1} \\
& \text{subject to}
& & \frac{\partial \sigma }{\partial t} + \nabla_{\bz} \cdot( (\bA \bz + \bB\bv )\sigma) =0, \label{FluidsVersionFL2} \\ 
& & & \sigma(\bx,t=0) = \sigma_{0}, \quad \sigma(\bx,t=1) = \sigma_{1},\label{FluidsVersionFL3} 
\end{align}
\end{subequations}
where
\begin{equation} \label{CostFunctional}
    \mathcal{L}(\bz,\bv):= \lVert \bdelta_{\btau }(\bz) + \bgamma_{\btau }(\bz) \bv \rVert_{2}^2.
\end{equation}
The infimum in (\ref{FluidsVersionFL}) is taken over the pair of transformed PDFs and admissible controls $(\sigma,\bv) \in \mathcal{P}(\mathcal{Z}) \times \mathcal{V}$ where $\mathcal{V}:=\{v: \mathcal{Z} \times [0,1] \mapsto \mathbb{R}^m|\:  \lVert v \rVert_2^2 < \infty\} $.

\begin{remark} \label{RemarkBack2Original}
The solution pair $(\rho^{{\rm{opt}}},\bu^{{\rm{opt}}})$ for (\ref{FluidDynamicsVersion}) can be recovered from the optimal solution $(\sigma^{{\rm{opt}}},\bv^{{\rm{opt}}})$ of (\ref{FluidsVersionFL}) via the transformations 
\begin{subequations}
    \begin{align}
    \rho^{{\rm{opt}}}(\bx,t)&= \sigma^{{\rm{opt}}}(\btau(\bx),t) \lvert {\rm{det}} \nabla_{\bx} \btau_{\bx}(\bx) \rvert, \\ 
    \bu^{{\rm{opt}}}(\bx,t)&= \bdelta(\bx) +  \bgamma(\bx)   \bv^{{\rm{opt}}}(\btau^{-1}(\bx),t)\
    \end{align}
\label{OriginalRecovery}    
\end{subequations}
for $\bx \in \mathcal{X}$, and $t\in [0,1]$. 
\end{remark}

\begin{example} 
To illustrate the reformulation (\ref{FluidsVersionFL}), let us reconsider the system (\ref{MIMOExample}). In this case, the inverse mapping of (\ref{ExampleChangeofCoordinates}) is given by
\begin{equation}
    \!\!\bx = \btau^{-1}\!(\bz):= \begin{pmatrix}
    z_1+z_5 & z_2 & z_3 + z_1z_4 & z_4 & z_5
    \end{pmatrix}^{\!\top}.
\end{equation} 
Here, the determinant of the Jacobian of (\ref{ExampleChangeofCoordinates}) is non-zero for all vectors in $\mathbb{R}^{n}$, i.e., $\mathcal{Z}=\mathbb{R}^{n}$.
From (\ref{ExampleFeedbackLaw}) and (\ref{defCompositeMap}), we have 
\begin{equation}\label{defdeltatau}
    \bdelta_{\btau}(\bz) = \begin{pmatrix}
    -z_2^2/\cos(z_1) \\ z_2^2
    \end{pmatrix},
\end{equation}
and 
\begin{equation}\label{defGammatau}
    \bgamma_{\btau}(\bz) = \begin{pmatrix}
    1/\cos(z_1) &  -1/\cos(z_1) \\ 0 & 1
    \end{pmatrix}.
\end{equation}
The functional $\mathcal{L}(\bz,\bv)$ in (\ref{CostFunctional}) equals 
\begin{align}\label{ExampleLagrangian}
    &\bv^{\top}\begin{pmatrix} 2/\cos^2(z_1) &-1/\cos(z_1) \nonumber\\ 
-1/\cos(z_1) & 1 \end{pmatrix} \bv + \langle 
                         \left(z_2^2/\cos^2(z_1), \right.\\
                         &\left.-z_2^2/\cos(z_1) + z_2^2
                       \right)^{\top}, \bv\rangle + z_2^4/\cos^2(z_1) + z_2^4.
\end{align}
\end{example}

\begin{remark} Because feedback linearization guarantees that the matrix pair $(\bA,\bB)$ is controllable, any vector $\bz_1 \in \mathcal{Z}$ is reachable from any other vector $\bz_0 \in \mathcal{Z}$ for all $t\in[0,1]$ via the flow of (\ref{LinearizedODE}). This ensures that in (\ref{FluidsVersionFL3}), the initial PDF $\sigma_0(\bz)$ can be steered to $\sigma_1(\bz)$ via the flow $\sigma(\bz,t)$ of the controlled Liouville PDE (\ref{FluidsVersionFL2}). Thus, the constraint set of (\ref{FluidsVersionFL}) is non-empty, and the problem is feasible.
\end{remark}

\subsection{Optimality} \label{optimality}
To show the existence and uniqueness of minimizer for (\ref{FluidsVersionFL}), we set $\bm{m
}:=\sigma \bv$, and consider the change of variable $(\sigma,\bv) \mapsto (\sigma,\bm{m})$, transforming (\ref{FluidsVersionFL}) into 
\begin{subequations} \label{FluidsVersionFLm}
\begin{align} 
& \underset{\sigma,\bm{m}} {\inf}
& & \int_{0}^{1}\int_{\mathcal{Z}} \mathcal{J}(\sigma,\bm{m} ) \: \differential \bz \: \differential t,\label{FluidsVersionFL1m} \\
& \text{subject to}
& & \frac{\partial \sigma }{\partial t} + \nabla_{\bz} \cdot( \bA \bz \sigma+ \bB\bm{m} ) =0, \label{FluidsVersionFL2m} \\ 
& & & \sigma(\bx,t=0) = \sigma_{0}, \quad \sigma(\bx,t=1) = \sigma_{1},\label{FluidsVersionFL3m} 
\end{align}
\end{subequations}
where 
\begin{equation}
\mathcal{J}(\sigma,\bm{m}):=
\begin{cases}
\frac{1}{2}\lVert \bdelta_{\btau }(\bz) + \bgamma_{\btau }(\bz) \frac{\bm{m}}{\sigma} \rVert_{2}^2\: \sigma 
& \text{if $\sigma >0$,} \\
0 & \text{if ($\sigma,\bm{m})$ = (0,0), } \\
+\infty & \text{otherwise}.
\end{cases}
\end{equation}
We note that $\mathcal{J}(\sigma,\bm{m})$ is the perspective function of the strictly convex map $\bm{m} \mapsto \lVert \bdelta_{\btau }(\bz)\sigma + \bgamma_{\btau }(\bz) \bm{m} \rVert_{2}^2$; therefore, $\mathcal{J}$ is jointly strictly convex in $(\sigma,\bm{m})$. The constraints (\ref{FluidsVersionFL2m})-(\ref{FluidsVersionFL3m}) are linear in $(\sigma,\bm{m})$. Hence, (\ref{FluidsVersionFLm}) admits a unique minimizing pair, and equivalently, so does (\ref{FluidsVersionFL}). The following theorem summarizes how this optimal pair for (\ref{FluidsVersionFL}), denoted hereafter as $(\sopt,\vopt)$, can be obtained.

\begin{theorem} \label{ThmOptimalControl}
(\textbf{Optimal control for (\ref{FluidsVersionFL})})
The optimal control $\vopt$ for the problem (\ref{FluidsVersionFL}), is given by
\begin{equation} \label{optcntrl}
    \bv^{{\rm{opt}}}(\bz,t) = (\bgamma_{\btau }^{\top}\bgamma_{\btau }(\bz))^{-1} \bB^{\!\top}\nabla_{\bz}\psi - \bgamma^{-1}_{\btau }(\bz) \bdelta_{\btau }(\bz),
\end{equation}
\end{theorem}
where $\psi$ solves the Hamilton-Jacobi-Bellman (HJB) PDE 
\begin{align}
    & \frac{\partial \psi}{\partial t}+\langle \nabla_{\bz}\psi, \bA \bz \rangle-\langle \nabla_{\bz}\psi,\bB \bgamma^{-1}_{\btau}(\bz)\bdelta_{\btau}(\bz) \rangle 
       \nonumber\\ 
      &+  \frac{1}{2}\langle\nabla_{\bz}\psi, \bB \left (\bgamma_{\btau}^{\top}(\bz) \bgamma_{\btau}(\bz) \right)^{-1}\bB^{\top}\nabla_{\bz} \psi   \rangle = 0.
      \label{HJBPDE}
\end{align}
Furthermore, if the optimal joint state PDF $\sigma^{{\rm{opt}}}$ is a solution to the Liouville PDE
\begin{align}
        \frac{\partial \sigma^{{\rm{opt}}}}{\partial t} + \nabla_{\bz}\cdot \left( \left(\bA\bz + \bB \bv^{{\rm{opt}}} \right)\right) = 0,
\label{LiouvillePDE}        
\end{align}
with boundary conditions $\sopt(\bz,0) =  \sopt_{0}(\bz)$, and $\sopt(\bz,1) =  \sopt_1(\bz)$, then the pair $(\sopt,\vopt)$ solves the problem (\ref{FluidsVersionFL}).

\begin{proof}
The Lagrangian associated with (\ref{FluidsVersionFL}) is 
\begin{align}
    \mathscr{L}(\sigma,\psi,\bv) &=\int_{0}^{1}  \int_{\mathcal{Z}}\frac{1}{2} \mathcal{L}(\bz,\bv) \sigma(\bz,t) \:\differential \bz  \differential t \nonumber\\
    &+ \underbrace{\int_{\mathcal{Z}} \int_{0}^{1} \psi(\bz,t)  \frac{\partial \sigma}{\partial t} \differential t \: \differential \bz}_{\rm{term \: 1}} \nonumber\\
   &\hspace{-0.2in}+\underbrace{\int_{0}^{1}  \int_{\mathcal{Z}} 
   \psi(\bz,t) \nabla_{\bz} \cdot \left( \left(\bA \bz + \bB \bv  \right)\sigma \right) \: \differential \bz  \differential t}_{\rm{term\: 2}}.
   \label{Lagrangian1}
    \end{align}
In (\ref{Lagrangian1}), we interchange the order of integration and  perform integration by parts w.r.t. $t$ in term 1, and w.r.t. $\bz$ in term 2. Since $\sigma(\bz,t) \rightarrow 0 $ as $z \rightarrow \partial \mathcal{Z}$, we can express $\mathscr{L}$ as 
\begin{equation} \label{Lagrangian2}
    \int_{0}^{1}  \int_{\mathcal{Z}} \left\{ \frac{1}{2}\mathcal{L}(\bz,\bv) - \frac{\partial \psi }{\partial t}- 
    \langle \nabla_{\bz} \psi, \bA \bz + \bB \bv \rangle    \right\} \sigma(\bz,t) \differential \bz \differential t.
\end{equation}
Performing pointwise minimization of the above w.r.t. $\bv$ while fixing $\sigma$, we obtain 
\begin{equation}
    \bgamma_{\btau}^{\top} \bgamma_{\btau}\bv^{\rm{opt}}(\bz,t) = \bB^{\top}\nabla_{z} \psi(\bz,t) -\bgamma_{\btau}^{\top}(\bz) \bdelta_{\btau}(\bz). 
    \end{equation}
Taking the matrix inverse on both sides yield (\ref{optcntrl}). Substituting $\bv^{\rm{opt}}$ back into (\ref{Lagrangian2}) and equating to zero, we then get 
\begin{equation}
\begin{aligned}
      &\int_{0}^{1}  \int_{\mathcal{Z}} \bigg\{
      -\frac{\partial \psi}{\partial t}- \langle \nabla_{\bz}\psi, \bA \bz \rangle+ \langle \nabla_{\bz}\psi,\bB \bgamma^{-1}_{\btau}(\bz)\bdelta_{\btau}(\bz) \rangle 
       \\ 
      &- \frac{1}{2}\langle\nabla_{\bz}\psi, \bB \left (\bgamma_{\btau}^{\top}(\bz) \bgamma_{\btau}(\bz) \right)^{\!-1}\!\!\bB^{\!\top}\nabla_{\bz} \psi   \rangle \bigg\} \sigma(\bz,t)\differential \bz \differential t = 0.
      \end{aligned}
\label{AlmostFinal}      
\end{equation}
Since (\ref{AlmostFinal}) holds for arbitrary $\sigma$, we arrive at (\ref{HJBPDE}). 
\end{proof}

\begin{example} (\textbf{HJB for (\ref{ExampleLagrangian})}) 
From (\ref{defdeltatau})-(\ref{defGammatau}), we have  
\begin{equation}\label{invofouterproduct}
    (\bgamma^{\top}_{\btau}(\bz)\bgamma_{\btau}(\bz))^{-1}     = \begin{pmatrix} \cos^2(z_1) &\cos(z_1) \\ 
\cos(z_1)  &2 \end{pmatrix}
\end{equation}
 and
 \begin{equation}\label{invGammataudeltatau}
     \bgamma_{\btau}^{-1}(\bz)\bdelta_{\btau}(\bz) = \begin{pmatrix}
     0 \\ z_2^2
     \end{pmatrix}.
 \end{equation}
Substituting (\ref{invofouterproduct})-(\ref{invGammataudeltatau}) into (\ref{HJBPDE}), and using the pair $(\bA,\bB)$ from (\ref{ExampleAB}), gives the HJB PDE 
\begin{align}
    &\frac{\partial \psi }{\partial t} 
    +z_2 \frac{\partial \psi}{\partial z_1} +z_3\frac{\partial \psi}{\partial z_2} + z_5\frac{\partial \psi}{\partial z_4} 
     - z_2^2 \frac{\partial \psi}{\partial z_5} \nonumber\\
    &+\frac{1}{2} \bigg[\!\cos^2(z_1)\!\left(\frac{\partial \psi}{\partial z_3}\right)^{\!\!2} \!\!+  \cos(z_1)\frac{\partial \psi}{\partial z_3}\frac{\partial \psi}{\partial z_5} + 2\left(\frac{\partial \psi}{\partial z_5}\right)^{\!\!2} \!\bigg] = 0.
    \label{ExampleHJB}
    \end{align}
\end{example}

\begin{remark}
Computing the pair $(\sigma^{{\rm{opt}}},\bv^{{\rm{opt}}})$ in Theorem \ref{ThmOptimalControl} is challenging in general since it calls for solving a system of coupled nonlinear PDEs (\ref{HJBPDE})-(\ref{LiouvillePDE}) with atypical boundary conditions. In the following Sections, we will provide further reformulations of (\ref{FluidsVersionFL}) to make it computationally amenable. 	
\end{remark}


\section{Stochastic Density Steering: Reformulation into Schr\"{o}dinger System}
\label{SecStochasticSteering}
Motivated by \cite{chen2017optimal}, we consider a generalized version of (\ref{FluidsVersionFL}) by adding a diffusion term to (\ref{FluidsVersionFL2}):
\begin{subequations} \label{FluidsVersionStochasticFL}
\begin{align} 
& \underset{\sigma,\bv} {\inf}
& & \int_{0}^{1}\int_{\mathcal{Z}} \frac{1}{2} \mathcal{L}(\bz,\bv) \sigma(\bz,t)  \: \differential \bz\:\differential t,\label{FluidsVersionStochasticFL1} \\
& \text{subject to}
& & \frac{\partial \sigma }{\partial t} + \nabla_{\bz} \cdot( (\bA \bz + \bB\bv )\sigma)  \nonumber\\
& & & \quad =\epsilon \bm{1}^{\top}\left(\bm{D}(\bz) \odot {\rm{Hess}}(\sigma) \right) \bm{1},\label{FluidsVersionStochasticFL2} \\ 
& & & \sigma(\bz,t=0) = \sigma_{0}, \quad \sigma(\bz,t=1) = \sigma_{1},\label{FluidsVersionStochasticFL3} 
\end{align}
\end{subequations}
where $\bD(\bz):= \bB \bgamma_{\tau} ^{-1}(\bz) (\bB \bgamma_{\tau}^{-1}(\bz))^{\top} $. In particular, the controlled Liouville PDE in (\ref{FluidsVersionFL2}) is now replaced by a Fokker-Planck-Kolmogorov PDE in (\ref{FluidsVersionStochasticFL2}), having an additional diffusion term $\sqrt{2\epsilon}\bB \bgamma_{\tau} ^{-1}(\bz)$, where the parameter $\epsilon>0$ (not necessarily small). Formally, this generalization is equivalent to adding a stochastic perturbation to the controlled sample path ODE $\dot{\bz} = \bA \bz + \bB \bv$, resulting in the It\^{o} SDE
\begin{equation}
    \differential \bz = (\bA \bz + \bB \bv) \: \differential t+ \sqrt{2\epsilon}\bB \bgamma_{\tau} ^{-1}(\bz) \: \differential  \bm{w},
\end{equation}
where $\bm{w}(t)\in\mathbb{R}^{m}$ is standard Wiener process. In the special case $\bdelta_{\tau}(\bz)\equiv 0 $ and $\bgamma_{\tau}(\bz)\equiv \bm{I}$, problem (\ref{FluidsVersionStochasticFL}) reduces to the Schr\"{o}dinger bridge problem with linear prior dynamics \cite[equation (49)]{chen2017optimal}. Thus, (\ref{FluidsVersionStochasticFL}) is a Schr\"{o}dinger bridge-like problem with a prior dynamics that has linear drift and nonlinear diffusion coefficient.

The following Theorem characterizes the minimizing pair $(\sopt,\vopt)$ for problem (\ref{FluidsVersionStochasticFL}).

\begin{theorem} \label{ThmOptimalStochasticControl}
(\textbf{Optimal control for (\ref{FluidsVersionStochasticFL})})
The optimal control $\bv^{{\rm{opt}}}(\bz,t)$ for (\ref{FluidsVersionStochasticFL}) is given by (\ref{optcntrl}), where $\psi$ solves the HJB PDE 
\begin{align}\label{StochasticHJBPDE}
    & \frac{\partial \psi}{\partial t}+\langle \nabla_{\bz}\psi, \bA \bz \rangle-\langle \nabla_{\bz}\psi,\bB \bgamma^{-1}_{\btau}(\bz)\bdelta_{\btau}(\bz) \rangle 
       \nonumber\\ 
      &+  \frac{1}{2}\langle\nabla_{\bz}\psi,  \bm{D}(\bz) \nabla_{\bz} \psi   \rangle  + \epsilon \langle \bm{D}(\bz), {\rm{Hess}}(\psi) \rangle = 0,
   \end{align}
and the optimal joint state PDF $\sigma^{{\rm{opt}}}(\bz,t)$ solves the controlled Fokker-Planck-Kolmogorov PDE 
\begin{align}
 \frac{\partial \sigma^{{\rm{opt}}} }{\partial t} &+ \nabla_{\bz} \cdot( (\bA \bz + \bB\bv^{{\rm{opt}}} )\sigma^{{\rm{opt}}})  \nonumber\\
 &-\epsilon \bm{1}^{\top}\left(\bm{D}(\bz) \odot {\rm{Hess}}(\sigma^{{\rm{opt}}}) \right) \bm{1}  = 0,
 \label{FPKPDE}
\end{align}
with boundary conditions 
\begin{equation}
    \sigma^{{\rm{opt}}}(\bz,t=0) = \sigma^{{\rm{opt}}}_{0}, \quad \sigma^{{\rm{opt}}}(\bz,t=1) = \sigma^{{\rm{opt}}}_{1}.
 \label{SBPBC}   
\end{equation}
\end{theorem}

\begin{proof}
The proof proceeds similarly as in Theorem \ref{ThmOptimalControl} except that we now have an additional term in the Lagrangian (\ref{Lagrangian1}) which we refer to as ``term 3", given by
\begin{equation} \label{term3}
-\underbrace{\epsilon\int_{0}^{1} \int_{\mathcal{Z}} \psi(\bz,t)  \bm{1}^{\top}\left(\bm{D}(\bz) \odot {\rm{Hess}}(\sigma^{{\rm{opt}}}) \right) \bm{1} \differential \bz \differential t.}_{{\rm{term \: 3}}}
\end{equation}
From the following chain of equalities:
\begin{align}  
& \int_{\mathcal{Z}} \langle \bm{D}(\bz), {\rm{Hess}}(\psi)  \rangle \sigma^{{\rm{opt}}}(\bz,t) \differential \bz \nonumber\\
&= \int_{\mathcal{Z}} \sum_{i,j=1}^n \bm{D}_{ij}(\bz) \dfrac{\partial \psi(\bz,t)}{\partial z_i \partial z_j} \sigma^{{\rm{opt}}}(\bz,t)\differential \bz \nonumber\\ 
&= \sum_{i,j=1}^n \int_{\mathcal{Z}} \bm{D}_{ij}(\bz) \dfrac{\partial \psi(\bz,t)}{\partial z_i \partial z_j} \sigma^{{\rm{opt}}}(\bz,t)\differential \bz  \nonumber\\
&= -\sum_{i,j=1}^n \int_{\mathcal{Z}}  \frac{\partial \psi(\bz,t) }{\partial z_j} \frac{\partial (\bm{D}_{ij}\sigma^{{\rm{opt}}}(\bz,t) ) }{\partial z_i} \differential \bz\nonumber\\
&=  \int_{\mathcal{Z}}  \psi(\bz,t) \sum_{i,j=1}^n \dfrac{\partial (\bm{D}(\bz)_{ij}\sigma^{{\rm{opt}}}(\bz,t))}{\partial z_j\partial z_i}  \differential \bz \nonumber\\ 
&= \int_{\mathcal{Z}} \psi(\bz,t)  \bm{1}^{\top}\left(\bm{D}(\bz) \odot {\rm{Hess}}(\sigma^{{\rm{opt}}}) \right) \bm{1}  \differential \bz
\label{ChainOfEqualitites}
\end{align}
we deduce that (\ref{term3}) is equal to 
\begin{equation}
    -\epsilon\int_{0}^{1} \int_{\mathcal{Z}} \langle \bm{D}(\bz), {\rm{Hess}}(\psi)  \rangle \sigma^{{\rm{opt}}}(\bz,t)  \differential \bz \differential t.
\end{equation}
So, the expression inside the curly braces in (\ref{Lagrangian2}), now will have an additional term $-\epsilon \langle \bm{D}(\bz), {\rm{Hess}}(\psi) \rangle$ that is independent of $\bv$. Therefore, pointwise  minimization of (\ref{Lagrangian2}) with this additional term
w.r.t. $\bv$, gives (\ref{optcntrl}), and the associated HJB PDE becomes (\ref{StochasticHJBPDE}). 
\end{proof}
Next, we show that the so-called Hopf-Cole transform \cite{hopf1950partial, cole1951quasi} allows to reduce the system of \emph{nonlinear} PDEs (\ref{StochasticHJBPDE})-(\ref{FPKPDE}) with boundary conditions (\ref{SBPBC}), into a system of boundary-coupled \emph{linear} PDEs, which we refer as the ``Schr\"{o}dinger System".

 \begin{theorem}\label{ThmSchrodingerSystem}
 (\textbf{Schr\"{o}dinger System}) Consider the Hopf-Cole transformation $(\sigma^{{\rm{opt}}},\psi) \mapsto (\varphi,\hat{\varphi})$: 
 \begin{subequations} \label{HopfCole}
     \begin{align}
     \varphi(\bz,t)&:= \exp(\psi(\bz,t) /2\epsilon), \label{HopfCole1}\\ 
     \hat{\varphi}(\bz,t) &:= \sigma^{{\rm{opt}}}(\bz,t) \exp(-\psi(\bz,t)/2\epsilon), \label{HopfCole2} 
     \end{align}
 \end{subequations}
 applied to the system of coupled nonlinear PDEs (\ref{StochasticHJBPDE})-(\ref{FPKPDE}). The pair $(\varphi,\hat{\varphi})$ satisfies the following system of linear PDEs:
 \begin{subequations} \label{ForwardBackward}
 \begin{align}
     \frac{\partial \varphi }{\partial t} &+ \langle \nabla_{\bz} \varphi, \bA \bz - \bB \bgamma_{\btau}^{-1} \bdelta_{\btau}(\bz) \rangle +  \epsilon \langle \bm{D},{{\rm{Hess}}}(\varphi)\rangle = 0,  \label{BackwardPDE} \\ 
\frac{\partial \hat{\varphi}}{\partial t} &+ \nabla_{\bz} \cdot \left( \left( \bA \bz - \bB \bgamma_{\btau}^{-1} \bdelta_{\btau}(\bz) \right)\hat{\varphi} \right) \notag \\
& \qquad \qquad -\epsilon \bm{1}\left( \bm{D}(\bz) \odot {\rm{Hess}}(\hat{\varphi})\right) \bm{1}=0, \label{ForwardPDE}
 \end{align}
 \end{subequations} 
with coupled boundary conditions 
\begin{equation}\label{CoupledBC}
    \varphi_0(\bz)\hat{\varphi}_0(\bz) = \sigma^{{\rm{opt}}}_0(\bz),\quad
     \varphi_1(\bz)\hat{\varphi}_1(\bz) = \sigma^{{\rm{opt}}}_1(\bz).
\end{equation}
\end{theorem}
\begin{proof}
From (\ref{HopfCole1}), $\psi = 2\epsilon \log \varphi$, which yields 
\begin{equation}
    \frac{\partial \psi }{\partial t} = \frac{ 2\epsilon }{\varphi}\frac{\partial \varphi}{\partial t}, \quad \nabla_{\bz}  \psi = \frac{2 \epsilon}{\varphi }\nabla_{\bz}  \varphi. 
\label{HopeCole1PartialDerivatives1}
\end{equation}
On the other hand, notice that  
\begin{align}   
&\quad\,\epsilon \langle \bm{D}(\bz) , {\rm{Hess}}(\psi) \rangle \nonumber\\
&=   \epsilon \langle \bm{D}(\bz) , {\rm{Hess}}(2\epsilon \log \varphi )\rangle \nonumber\\
&= 2\epsilon^2 \sum_{i,j=1}^{n} \bm{D}_{ij}(\bz) \frac{ \partial^{2} \log \varphi}{\partial z_i \partial z_j} \nonumber\\
&= 2\epsilon^2 \sum_{i,j=1}^{n} \bm{D}_{ij}(\bz) \left( \frac{1}{\varphi} \frac{\partial^2 \varphi }{\partial z_i z_j} - \frac{1}{\varphi^2} \frac{\partial \varphi }{\partial z_i} \frac{\partial \varphi }{\partial z_j}\right) \nonumber\\
&= \frac{2\epsilon^2}{\varphi} \langle \bm{D}(\bz), {\rm{Hess}}(\varphi) \rangle  - \frac{2 \epsilon^2}{\varphi^2} \langle \nabla_{\bz}  \varphi,\bm{D}(\bz) \nabla_{\bz}   \varphi \rangle.
\label{HopeCole1PartialDerivatives2}	
\end{align}
Substituting (\ref{HopeCole1PartialDerivatives1}) and (\ref{HopeCole1PartialDerivatives2}) into (\ref{StochasticHJBPDE}) yields 
\begin{align*}
    &\frac{ 2\epsilon }{\varphi}\frac{\partial \varphi}{\partial t} + \frac{2 \epsilon}{\varphi }\langle \nabla_{\bz}  \varphi, \bA \bz- \bB \bgamma_{\btau}^{-1}(\bz) \bdelta_{\btau}(\bz) \rangle   \nonumber\\ 
    & +\frac{1}{2} \frac{4\epsilon^2}{\varphi^2} \langle \nabla_{\bz} \varphi, \bm{D}(\bz) \nabla_{\bz} \varphi \rangle +\frac{2\epsilon^2}{\varphi} \langle \bm{D}(\bz), {\rm{Hess}}(\varphi) \rangle  \nonumber\\
    &-\frac{2 \epsilon^2}{\varphi^2} \langle \nabla_{\bz} \varphi,\bm{D}(\bz) \nabla  \varphi \rangle = 0,
\end{align*}
which gives (\ref{BackwardPDE}).

Next, let $\bm{\omega}(\bz) := \bA\bz-\bB\bgamma_{\tau}^{-1}(\bz) \bdelta_{\tau}(\bz)$. We then have
\begin{align}
    &\nabla_{\bz} \cdot (\hat{\varphi}\bm{\omega}(\bz)) = \langle \nabla_{\bz} \hat{\varphi},\bm{\omega}(\bz)  \rangle + \hat{\varphi} \nabla_{\bz} \cdot \bm{\omega}(\bz) \nonumber\\
    &=  \exp\left( -\psi / 2 \epsilon \right) \bigg( \langle \nabla \sopt , \bm{\omega}(\bz) \rangle + \sopt \nabla_{\bz} \cdot \left( \bm{\omega}(\bz) \right)\nonumber\\
    &- \frac{\sopt}{2\epsilon} \langle \nabla_{\bz} \psi,  \bm{\omega}(\bz) \rangle \bigg ),
 \label{Intermed1}    
\end{align}
and 
\begin{align}
   &\epsilon \bm{1}^{\top} (\bm{D} \odot {\rm{Hess}}(\hat{\varphi})) \bm{1} =  \exp\left(-\frac{\psi}{2 \epsilon}\!\right)\!\bigg(\!\epsilon \bm{1}^{\top} (\bm{D} \odot {\rm{Hess}}(\hat{\sopt})) \bm{1} \nonumber\\
   &-  \frac{\partial (\bm{D}_{ij}(\bz) \sopt )}{\partial x_j} \frac{\partial \psi}{\partial z_j} -\frac{1}{2} \sopt \langle \bm{D}(\bz), {\rm{Hess}}(\psi) \rangle  \nonumber\\
   &+\frac{\sopt}{4\epsilon} \langle \nabla_{\bz} \psi , \bm{D}(\bz) \nabla_{\bz} \psi \rangle \bigg).
   \label{Intermed2} 
    \end{align}
In (\ref{HopfCole2}), taking the partial derivative of $\hat{\varphi}$ w.r.t. $t$, and using (\ref{StochasticHJBPDE})-(\ref{FPKPDE}) together with (\ref{Intermed1})-(\ref{Intermed2}), we get 
\begin{align*}
    &\frac{\partial \hat{\varphi}}{\partial t} =  \exp\left( -\psi / 2 \epsilon \right)\bigg(\frac{\partial \sopt }{\partial t}- \frac{\sopt}{2\epsilon} \frac{\partial \psi}{\partial t}\bigg) \\
    & =  \exp\left( -\psi / 2 \epsilon \right)\bigg( -\nabla_{\bz}(\sopt \bw(\bz))  -\nabla_{\bz}(\sopt \bm{D} \nabla \psi) \\
    &+ \epsilon \bm{1}( \bm{D}(\bz) \odot {\rm{Hess}}(\sopt)) \bm{1} + \frac{\sopt}{2\epsilon}\langle \nabla_{\bz} \psi, \bw(\bz) \rangle  \\
    &+\frac{\sopt}{4\epsilon} \langle \nabla_{\bz} \psi, \bm{D}(\bz) \nabla_{\bz}\psi \rangle  +\frac{\sopt}{2} \langle \bm{D}(\bz) , {\rm{Hess}}(\psi) \rangle \bigg)\\
    &= -\nabla_{\bz}\cdot (\hat{\varphi} \bw(\bz)) + \epsilon \bm{1}^{\top} (\bm{D} \odot {\rm{Hess}}(\hat{\varphi})) \bm{1},
\end{align*}
which is indeed (\ref{ForwardPDE}). The boundary conditions (\ref{SBPBC}) follow directly from (\ref{HopfCole1})-(\ref{HopfCole2}).
\end{proof}

Theorem \ref{ThmSchrodingerSystem} in principle allows solving problem (\ref{FluidsVersionStochasticFL}) in the following manner. Let $(\varphi_{1},\hat{\varphi}_{0}) := \left(\varphi\left(\bz,t=1\right),\hat{\varphi}\left(\bz,t=0\right)\right)$ denote the terminal-initial condition pair for the system (\ref{BackwardPDE})-(\ref{ForwardPDE}). By making an arbitrary guess for the pair $(\varphi_{1},\hat{\varphi}_{0})$, one can perform a fixed point recursion on the Schr\"{o}dinger system (\ref{ForwardBackward})-(\ref{CoupledBC}), and the converged pair $(\varphi_{1},\hat{\varphi}_{0})$ can then be used to compute the transient pair $\left(\varphi\left(\bz,t\right),\hat{\varphi}\left(\bz,t\right)\right)$. Then, by (\ref{HopfCole}), we recover $(\sopt,\psi)$, and thus $(\sopt,\vopt)$ from (\ref{optcntrl}). Notice that this procedure with small $\epsilon>0$ will yield the pair $(\sopt,\vopt)$ solving problem (\ref{FluidsVersionFL}). Finally, the mapping (\ref{OriginalRecovery}) in Remark \ref{RemarkBack2Original} recovers the solution $(\ropt,\uopt)$ for problem (\ref{FluidDynamicsVersion}). This algorithmic framework and its convergence will be the topic of our future research.

\section*{Conclusions}
We considered the minimum energy joint state PDF steering problem over finite time horizon subject to the multi-input state feedback linearizable dynamics. We showed that the density steering problem can be made amenable in the feedback linearized coordinates. We derived the state feedback controller in terms of the solutions of a pair of coupled HJB and Fokker-Planck-Kolmogorov PDEs. Furthermore, we reduced this system of coupled nonlinear PDEs to a system of boundary-coupled linear PDEs. Our results are expected to lay the foundation for developing computational algorithms solving the density steering problem.



\bibliographystyle{IEEEtran}
\bibliography{referencesMIMOacc2020.bib}

\begin{thebibliography}{10}
\providecommand{\url}[1]{#1}
\csname url@samestyle\endcsname
\providecommand{\newblock}{\relax}
\providecommand{\bibinfo}[2]{#2}
\providecommand{\BIBentrySTDinterwordspacing}{\spaceskip=0pt\relax}
\providecommand{\BIBentryALTinterwordstretchfactor}{4}
\providecommand{\BIBentryALTinterwordspacing}{\spaceskip=\fontdimen2\font plus
\BIBentryALTinterwordstretchfactor\fontdimen3\font minus
  \fontdimen4\font\relax}
\providecommand{\BIBforeignlanguage}[2]{{%
\expandafter\ifx\csname l@#1\endcsname\relax
\typeout{** WARNING: IEEEtran.bst: No hyphenation pattern has been}%
\typeout{** loaded for the language `#1'. Using the pattern for}%
\typeout{** the default language instead.}%
\else
\language=\csname l@#1\endcsname
\fi
#2}}
\providecommand{\BIBdecl}{\relax}
\BIBdecl

\bibitem{bandyopadhyay2017probabilistic}
S.~Bandyopadhyay, S.-J. Chung, and F.~Y. Hadaegh, ``Probabilistic and
  distributed control of a large-scale swarm of autonomous agents,'' \emph{IEEE
  Transactions on Robotics}, vol.~33, no.~5, pp. 1103--1123, 2017.

\bibitem{deshmukh2018mean}
V.~Deshmukh, K.~Elamvazhuthi, S.~Biswal, Z.~Kakish, and S.~Berman, ``Mean-field
  stabilization of markov chain models for robotic swarms: Computational
  approaches and experimental results,'' \emph{IEEE Robotics and Automation
  Letters}, vol.~3, no.~3, pp. 1985--1992, 2018.

\bibitem{monga2018synchronizing}
B.~Monga, G.~Froyland, and J.~Moehlis, ``Synchronizing and desynchronizing
  neural populations through phase distribution control,'' in \emph{2018 Annual
  American Control Conference (ACC)}.\hskip 1em plus 0.5em minus 0.4em\relax
  IEEE, 2018, pp. 2808--2813.

\bibitem{chertkov2017ensemble}
M.~Chertkov and V.~Chernyak, ``Ensemble of thermostatically controlled loads:
  Statistical physics approach,'' \emph{Scientific reports}, vol.~7, no.~1, p.
  8673, 2017.

\bibitem{halder2016architecture}
A.~Halder, X.~Geng, P.~Kumar, and L.~Xie, ``Architecture and algorithms for
  privacy preserving thermal inertial load management by a load serving
  entity,'' \emph{IEEE Transactions on Power Systems}, vol.~32, no.~4, pp.
  3275--3286, 2016.

\bibitem{okamoto2019optimal}
K.~Okamoto and P.~Tsiotras, ``Optimal stochastic vehicle path planning using
  covariance steering,'' \emph{IEEE Robotics and Automation Letters}, vol.~4,
  no.~3, pp. 2276--2281, 2019.

\bibitem{chen2015optimal1}
Y.~Chen, T.~T. Georgiou, and M.~Pavon, ``Optimal steering of a linear
  stochastic system to a final probability distribution, part i,'' \emph{IEEE
  Transactions on Automatic Control}, vol.~61, no.~5, pp. 1158--1169, 2015.

\bibitem{chen2015optimal2}
------, ``Optimal steering of a linear stochastic system to a final probability
  distribution, part ii,'' \emph{IEEE Transactions on Automatic Control},
  vol.~61, no.~5, pp. 1170--1180, 2015.

\bibitem{chen2017optimal}
------, ``Optimal transport over a linear dynamical system,'' \emph{IEEE
  Transactions on Automatic Control}, vol.~62, no.~5, pp. 2137--2152, 2017.

\bibitem{chen2016relation}
------, ``On the relation between optimal transport and {Schr{\"o}dinger}
  bridges: A stochastic control viewpoint,'' \emph{Journal of Optimization
  Theory and Applications}, vol. 169, no.~2, pp. 671--691, 2016.

\bibitem{halder2016finite}
A.~Halder, and E.D.B.~Wendel, ``Finite horizon linear quadratic Gaussian density regulator with Wasserstein terminal cost," \emph{Proceedings of the 2016 American Control Conference (ACC)}, pp. 7249--7254, 2016. 

\bibitem{villani2003topics}
C.~Villani, \emph{Topics in optimal transportation}.\hskip 1em plus 0.5em minus
  0.4em\relax American Mathematical Soc., 2003, no.~58.

\bibitem{villani2008optimal}
------, \emph{Optimal transport: old and new}.\hskip 1em plus 0.5em minus
  0.4em\relax Springer Science \& Business Media, 2008, vol. 338.

\bibitem{schrodinger1931umkehrung}
E.~Schr{\"o}dinger, ``{\"U}ber die {\"u}mkehrung der naturgesetze,''
  \emph{Verlag der Akademie der Wissenschaften, in Kommission bei Walter de
  Gruyter}, 1931.

\bibitem{schrodinger1932theorie}
------, ``Sur la th{\'e}orie relativiste de l'{\'e}lectron et
  l'interpr{\'e}tation de la m{\'e}canique quantique,'' \emph{Ann. Inst. H.
  Poincar{\'e}}, vol.~2, no.~4, pp. 269--310, 1932.

\bibitem{leonard2014survey}
C.~L{\'e}onard, ``A survey of the {Schr{\"o}dinger} problem and some of its
  connections with optimal transport,'' \emph{Discrete \& Continuous Dynamical
  Systems-A}, vol.~34, no.~4, pp. 1533--1574, 2014.

\bibitem{hotz1987covariance}
A.~Hotz and R.~E. Skelton, ``Covariance control theory,'' \emph{International
  Journal of Control}, vol.~46, no.~1, pp. 13--32, 1987.

\bibitem{skelton1989covariance}
R.~E. Skelton and M.~Ikeda, ``Covariance controllers for linear continuous-time
  systems,'' \emph{International Journal of Control}, vol.~49, no.~5, pp.
  1773--1785, 1989.

\bibitem{zhu1995covariance}
G.~Zhu, K.~M. Grigoriadis, and R.~E. Skelton, ``Covariance control design for
  hubble space telescope,'' \emph{Journal of Guidance, Control, and Dynamics},
  vol.~18, no.~2, pp. 230--236, 1995.

\bibitem{skelton1997unified}
R.~E. Skelton, T.~Iwasaki, and D.~E. Grigoriadis, \emph{A unified algebraic
  approach to control design}.\hskip 1em plus 0.5em minus 0.4em\relax CRC
  Press, 1997.

\bibitem{okamoto2018optimal}
K.~Okamoto, M.~Goldshtein, and P.~Tsiotras, ``Optimal covariance control for
  stochastic systems under chance constraints,'' \emph{IEEE Control Systems
  Letters}, vol.~2, no.~2, pp. 266--271, 2018.

\bibitem{agrachev2009optimal}
A.~Agrachev and P.~Lee, ``Optimal transportation under nonholonomic
  constraints,'' \emph{Transactions of the American Mathematical Society}, vol.
  361, no.~11, pp. 6019--6047, 2009.

\bibitem{elamvazhuthi2018optimal}
K.~Elamvazhuthi, P.~Grover, and S.~Berman, ``Optimal transport over
  deterministic discrete-time nonlinear systems using stochastic feedback
  laws,'' \emph{IEEE control systems letters}, vol.~3, no.~1, pp. 168--173,
  2018.

\bibitem{isidori1989nonlinear}
A.~Isidori, \emph{Nonlinear control systems}, ser. Lecture Notes in Control and
  Information Sciences.\hskip 1em plus 0.5em minus 0.4em\relax Springer-Verlag,
  1989.

\bibitem{benamou2000computational}
J.-D. Benamou and Y.~Brenier, ``A computational fluid mechanics solution to the
  {Monge-Kantorovich} mass transfer problem,'' \emph{Numerische Mathematik},
  vol.~84, no.~3, pp. 375--393, 2000.

\bibitem{brockett2007optimal}
R.~W. Brockett, ``Optimal control of the {Liouville} equation,'' \emph{AMS IP
  Studies in Advanced Mathematics}, vol.~39, pp. 23--35, 2007.

\bibitem{caluya2019finite}
K.~F. Caluya and A.~Halder, ``Finite horizon density control for static state
  feedback linearizable systems,'' \emph{arXiv preprint arXiv:1904.02272},
  2019.

\bibitem{hopf1950partial}
E.~Hopf, ``The partial differential equation $u_{t} + uu_{x}= \mu u_{xx}$,''
  \emph{Communications on Pure and Applied mathematics}, vol.~3, no.~3, pp.
  201--230, 1950.

\bibitem{cole1951quasi}
J.~D. Cole, ``On a quasi-linear parabolic equation occurring in aerodynamics,''
  \emph{Quarterly of applied mathematics}, vol.~9, no.~3, pp. 225--236, 1951.

\end{thebibliography}

\end{document}